\documentclass[11pt]{article}
\usepackage{amsfonts, amsmath, amsthm, amssymb, latexsym, amscd, color}
\usepackage{tabularx,ragged2e,booktabs}
\usepackage{rotating}
\usepackage{mathrsfs,calrsfs}
\usepackage{textcomp}
\usepackage{graphicx,tikz}
\newcommand{\comment}[1]{}
\newif\ifpdf
\ifpdf \else \fi \textwidth = 6.5 in \textheight = 9 in
\newtheorem{thm}{Theorem}[section]

\newtheorem{corollary}[thm]{Corollary}
\newtheorem{lemma}[thm]{Lemma}
\newtheorem{definition}[thm]{Definition}
\newtheorem{theorem}[thm]{Theorem}

\newtheorem{remark}[thm]{Remark}

\begin{document}

\title{Connectivity of 2-distance graphs}
\author{{\small S.H. Jafari$^{a}$, S.R. Musawi}$^{b}$\\
\\{\small $^{a}$Faculty of Mathematical Sciences, Shahrood University of Technology,}\\ {\small  P.O. Box 36199-95161, Shahrood, Iran}\\
{\small Email: shjafari@shahroodut.ac.ir, shjafari55@gmail.com}\\
\\{\small $^{b}$Faculty of Mathematical Sciences, Shahrood University of Technology,}\\ {\small  P.O. Box 36199-95161, Shahrood, Iran}\\
{\small Email: r\_musawi@shahroodut.ac.ir, r\_musawi@yahoo.com}\\
}
\date{}
\maketitle

\begin{abstract}
For a simple graph $G$, the $2$-distance graph, $D_2(G)$, is a graph with the  vertex set $V(G)$ and two vertices are adjacent if and only if their distance is $2$ in the graph $G$. In this paper, we characterize all graphs with connected 2-distance graph. For graphs with diameter 2, we prove that $D_2(G)$ is connected if and only if $G$ has no spanning complete bipartite subgraphs. For graphs with  a diameter greater than 2, we define a maximal   Fine set and by contracting $G$ on these subsets, we get a new graph $\hat G$  such that $D_2(G)$ is connected if and only if $D_2(\hat G)$ is connected. Especially, $D_2(G)$ is disconnected if and only if $\hat G$ is bipartite.
\end{abstract}


\section{introduction}

In this paper, we only consider the undirected finite simple graphs $G = (V,E)$.
 The distance between two vertices of $x,y\in V(G)$, $d(x,y)$, is the length of the shortest path between them. The diameter of $G$, $diam(G)$, is the maximum distance between vertices of $G$.
The graph $G$ is a bipartite graph if $V(G)$ be the union of two disjoint non-empty sets $A$ and $B$ where every edge connects a vertex of $A$ to a vertex of $B$.
It is well known that:
A graph $G$ is bipartite if and only if $G$ has no odd cycle.
The contracted graph of the partition $\{A_1, \cdots , A_t\}$ of  $V(G)$ is the graph $G_1$ with vertex set $V(G_1)=\{A_1, \cdots , A_t\}$ and two vertices $A_i$ and $A_j$  are adjacent if there is an edge $xy\in E(G)$ such that $x\in A_i$ and $y\in A_j$.

For a positive integer $k$, the $k$-power of $G$, denoted by $G^k$, is the graph with the vertex set $V(G^k)=V(G)$ and two vertices $x$ and $y$ are adjacent if and ony if $d(x,y)\leqslant k$.
$2$-distance graph of a graph $G$, $D_2(G)$, is a graph with the  vertex set $V(G)$ and two vertices are adjacent if and only their distance is $2$ in the graph $G$. 
One can see that for any graph $G$, $E(D_2(G))=E(G^2)-E(G)$.

The idea of 2-distance graphs is a particular case of a general notion of k-distance
graphs, which was first studied by Harary et al. \cite{hhk}. They investigated the connectedness
of 2-distance graph of some graphs. In the book by Prisner  \cite[157-159]{p}, the dynamics
of k-distance operator was explored. Furthermore, Boland et al. \cite{bhl} extended the
k-distance operator to a graph invariant they call distance-n domination number.
To solve one problem posed by Prisner \cite[Open Problem 29]{p}, Zelinka
\cite{z} constructed a graph that is r-periodic under the 2-distance operator for each
positive even integer r. Prisner’s problem was completely solved by Leonor in her
master’s thesis \cite{l}, wherein she worked with a graph construction different from
that of Zelinka’s.
Recently, Azimi and Farrokhi \cite{af2} studied all graphs whose 2-distance graphs
have maximum degree 2. They also solved the problem of finding all graphs whose
2-distance graphs are paths or cycles. More recently, the same authors \cite{af1} determined
all finite simple graphs with no square, diamond, or triangles with a common
vertex that are self 2-distance graphs (that is, graphs G satisfying the equation
$D^2(G)=G$. They further showed the nonexistence of cubic self 2-distance
graphs.
In \cite{c}, Ching   gives some characterizations of 2-distance graphs and finds all graphs X such that $D_2(X)=kP_2$
or $K_m \cup K_n$, where $k \geqslant 2$ is an integer, $P_2$ is the path of order $2$, and $K_m$ is the complete graph of order $m \geqslant 1$. 
In \cite{k} Khormali gives some conditions for connectivity $D_k(G)$. For $D_2(G)$, he proves that if $G$ has no odd cycle, then $D_2(G)$ is disconnected.

In this paper, we charactrize all graphs with connected 2-distance graph.
For graphs with diameter 2, we prove that $D_2(G)$ is connected if and only if $G$ has no spanning complete bipartite subgraphs. For graphs with diameter greater than 2, we define a maximal   Fine set and by contracting $G$ on this subsets, we get a new graph $\hat G$  such that $D_2(G)$ is connected if and only if $D_2(\hat G)$ is connected. Specially, $D_2(G)$ is disconnected if and only if $\hat G$ is bipartite.

\section{Cinnectivity of $D_2(G)$}

Let $G$ be a graph with $diam(G)=2$, then one can see that $D_2(G)=G^{\prime}$. Therefore, the connectivity of $D_2(G)$ come back to the connectivity of $G^{\prime}$. Thus, we have the following theorem which charactrized all gaphs with diameter 2 and connected 2-distance graph.
\begin{theorem}
Let $G$ be a graph with $diam(G)=2$, then $D_2(G)$ is disconnected if and only if $G$ has a spannig complete bipartite subgraph.
\end{theorem}

By above theorem, from now on we focus on connected graphs with diameter at least $3$.
Fiest, we introduce a fmily of graphs with connected $2$-distance graph which play the main role in our proofs.

\begin{definition}
The family $\mathcal{H}$ of graphs contains all odd cycles with degree greater than 4 and following three special graphs.
\end{definition}
\begin{figure}[h]\label{fig0}
\begin{center}
\begin{tikzpicture}
\draw (1,0)--(1,1);
\draw (2,0)--(1,0)--(0,0.5)--(1,1)--(2,1);
\draw (5,0)--(5,1)--(4,0.5)--(5,0)--(6,0)--(6,1)--(5,1);
\filldraw (0,0.5) circle(2pt);\filldraw (1,0) circle(2pt);\filldraw (2,0) circle(2pt);
\filldraw (2,1) circle(2pt);\filldraw (1,1) circle(2pt);
\filldraw (5,0) circle(2pt);\filldraw (5,1) circle(2pt);\filldraw (4,0.5) circle(2pt);
\filldraw (6,0) circle(2pt);\filldraw (6,1) circle(2pt);
\draw(8,0)--(9,0)--(10,0)--(11,0)--(10.5,-1)--(10,0);
\draw (8,0)--(9.5,1)--(9,0);\draw (10,0)--(9.5,1)--(11,0);
\filldraw (8,0) circle(2pt);\filldraw (9,0) circle(2pt);\filldraw (10,0) circle(2pt);
\filldraw (11,0) circle(2pt);\filldraw (9.5,1) circle(2pt);
\filldraw (10.5,-1) circle(2pt);
\end{tikzpicture}
\end{center}
\end{figure}

The following results is trivial.
\begin{lemma}\label{23}
\begin{enumerate}
\item
If $G\in\mathcal{H}$, then $D_2(G)$ is connected.
\item
If $G$  be a connected graph and $|V(G)|\leqslant4$, then $D_2(G)$ is disconnected.
\item
Let $G$  be a connected graph and $|V(G)|=5$. $D_2(G)$ is connected iff $G\in\mathcal{H}$.
\end{enumerate}
\end{lemma}

\begin{lemma}\label{24}
Let $G$ be a connected graph. If $H\in\mathcal{H}$ be an induced subgraph of $G$  such that $V(H)\nsubseteq N(x)$ for any vertex $x\in V(G)$, then $D_2(G)$ is connected.
\end{lemma}
\begin{proof}
Since $D_2(H)$ is connected and $H$ is an induced subgraph, then for any two non-adjacent vertices $x,y\in V(H)$, we have $d(x,y)\geqslant2$ and thus $D_2(H)$ is a subgraph of $D_2(G)$. Let $u\in V(G)\setminus V(H)$, $d(u, V(H))=t$ and $u=u_0u_1\cdots u_t=v$ is a path with $v\in V(H)$. Since $V(H)\not\subseteq N(u_{t-1})$, then there is $w\in V(H)$ such that $d(u_{t-1},w)=2$. Consequently, $u$ is connected to $v$ or $w$ in $D_2(G)$ as required.
\end{proof}

\begin{definition}
Let $A$ and $B$ are two distinct set of vertices. $[A,B]$ denotes the complete bipartite graph constructed by $A$ and $B$ as the independent sets.
\end{definition}

\begin{lemma}\label{999}
Let $G$ be a graph such that $D_2(G)$ is disconnected. If $H$ be a connected induced subgraph of $G$ and $D_2(H)$ is connected, then there is a subgraph $K$ of $G$ such that $H\leqslant K$, $D_2(K)$ is connected and   $[V(K),N_G(V(K))\setminus V(K)]$ is a subgraph of $G$.
\end{lemma}
\begin{proof}
If $[V(H),N_G(V(H))\setminus V(H)]$ is a subgraph of $G$, then $K=H$. Otherwise, there is $x\in N_G(V(H))\setminus V(H)$ such that $V(H)\not\subseteq N_G(x)$. Similar to proof of Lemma \ref{24}, $D_2(\langle V(H)\cup{x}\rangle)$ is connected. By setting $H_1=\langle V(H)\cup{x}\rangle$, we can continue this process by replacing $H_1$ to $H$. Since $G$ is finite there $H_t$ such that $H_t\ne G$ and for any $x\in N_G(V(H_t))\setminus V(H_t)$, $V(H_t)\subseteq N_G(x)$. By setting $K=H_t$, $[V(K),N_G(V(K))\setminus V(K)]$ is a subgraph of $G$, which completes the proof.
\end{proof}

Due to the above lemma, we introduce the following definition:

\begin{definition}
Let $G$ be a graph and $A\subseteq V(G)$.  The set $A$ is called  \textbf{Fine} if $[V(A),N_G(V(A))\setminus V(A)]$ is a subgraph of $G$, or equivalently for any $x\in A$ , $N(x)\setminus A=N(A)\setminus A\ne\emptyset$ .
\end{definition}

\begin{lemma}
For any $x\in V(G)$ with $diam(G)\geqslant3$, there is a unique  maximal Fine set $A$ such that $x \in A$.
\end{lemma}
\begin{proof}
For $x\in V(G)$, $\{x\}$ is Fine. Since $G$ is finite, there are some maximal Fine sets containing $x$. 
Let $A$ and $B$ two maximal Fine sets containing $x$. we claim that $A\cup B$ is Fine. By hypothesis $N(x)\setminus A=N(A)\setminus A$ and $N(x)\setminus B=N(B)\setminus B$. \\
First, we prove that $N(A\cup B)\setminus (A\cup B)\ne\emptyset$. \\
If $N(A\cup B)\subseteq  (A\cup B)$, since $G$ is connected, $A\cup B=V(G)$ and   $N(B)\setminus B\subseteq A$ and $N(A)\setminus A\subseteq B$. Let $y\in N(B)\setminus B$, then $y\in A$, and since $B$ is Fine, then  $B\subseteq N(y)\subseteq N(A)$. Thus, $B\setminus A\subseteq N(A)\setminus A$ and therefore $B\setminus A=N(A)\setminus A$. Therefore $diam(G)=2$ because $A$ is Fine, a contradiction.

Since $x\in A\cup B$ then it is trivial that $N(x)\setminus (A\cup B)\subseteq N(A\cup B)\setminus (A\cup B)$. Now let $y\in N(A\cup B)\setminus (A\cup B)$ then $y\not\in A\cup B$ and $y\in N(A\cup B)$. So $y\in N(A)$ or $y\in N(B)$. 
Without loss of generality, let $y\in N(A)$ then $y\in N(A)\setminus A=N(x)\setminus A$, hence $y\in N(x)\setminus (A\cup B)$.
Therefore, $ N(x)\setminus (A\cup B)= N(A\cup B)\setminus (A\cup B)$. On the other hand, for any $a\in A$, $ N(a)\setminus A= N(A)\setminus A=N(x)\setminus N(A)$ 
and so $ N(a)\setminus (A\cup B)= N(x)\setminus (A\cup B)= N(A\cup B)\setminus (A\cup B)$. Similarly, for any $b\in B$, $ N(b)\setminus (A\cup B)= N(A\cup B)\setminus (A\cup B)$. Consequently, $A\cup B$ is a Fine set and by maximality $A$ and $B$, $A=A\cup B=B$.
\end{proof}

\begin{remark}
For a connected graph with diameter 2, it may be the maximal Fine set containing a vertex do not  be unique. For example, in Cycle graph $C_4$, the maximal Fine set corresponding to each vertex is unique.

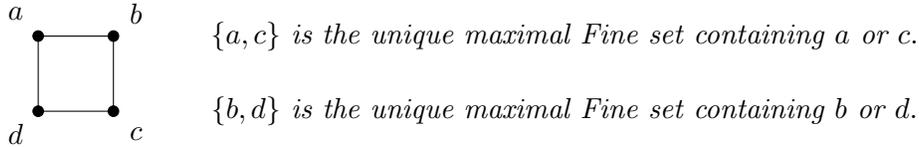
\vspace{-.7cm}
\begin{center}
\begin{tikzpicture}
\draw (0,0)--(0,1)--(1,1)--(1,0)--(0,0);
\node at (-.3,-.3){$d$};\node at (-.3,1.3){$a$};\node at (1.3,1.3){$b$};\node at (1.3,-.3){$c$};
\node at (7,1){$\{a,c\}\text{ is the unique maximal Fine set containing }a \text{ or } c.$};
\node at (7,0){$\{b,d\}\text{ is the unique maximal Fine set containing }b \text{ or } d.$};
\filldraw (0,0) circle(2pt);\filldraw (0,1) circle(2pt);\filldraw (1,0) circle(2pt);\filldraw (1,1) circle(2pt);
\end{tikzpicture}
\end{center}
But in the following graph, the maximal containing a vertex is not unique. 
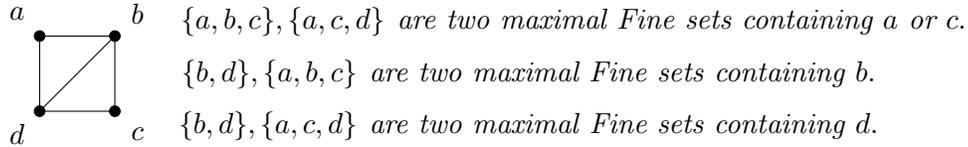
\begin{center}
\begin{tikzpicture}
\draw (1,1)--(0,0)--(0,1)--(1,1)--(1,0)--(0,0);
\node at (-.3,-.3){$d$};\node at (-.3,1.3){$a$};\node at (1.3,1.3){$b$};\node at (1.3,-.3){$c$};
\node at (7.1,1.2){$\{a,b,c\}, \{a,c,d\}\text{ are two maximal Fine sets containing }a \text{ or } c.$};
\node at (6.5,.5){$\{b,d\}, \{a,b,c\}\text{ are two maximal Fine sets containing }b.$};
\node at (6.5,-.2){$\{b,d\}, \{a,c,d\}\text{ are two maximal Fine sets containing }d.$};
\filldraw (0,0) circle(2pt);\filldraw (0,1) circle(2pt);\filldraw (1,0) circle(2pt);\filldraw (1,1) circle(2pt);
\end{tikzpicture}
\end{center}
\end{remark}

\begin{remark}
It seems that if a Fine set is not maximal, we can add a vertex to find a new Fine set contaning the previous Fine set. It is not true.\\ For example, in the graph shown in the following graph,  $\{a,b\}$ is Fine but it is not a maxmal Fine, because $\{a,b,c,x,y\}$ is maximal Fine containing both $a$ and $b$. But none of the 3-element sets contaning $\{a,b\}$ is not Fine.
\end{remark}

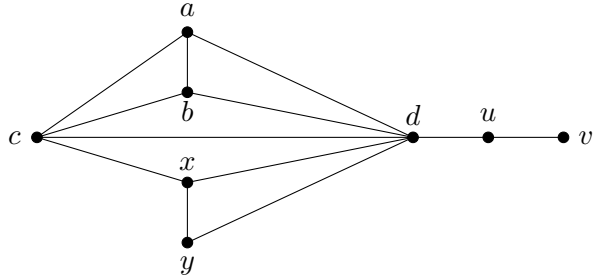
\begin{figure}[h] \label{7}
\begin{center}
\begin{tikzpicture}
\draw (-1,0)--(1,.6)--(1,1.4)--(4,0)--(1,-1.4)--(1,-.6)--(-1,0)--(5,0)--(6,0);
\draw (1,.6)--(4,0)--(1,-.6);\draw (1,1.4)--(-1,0);
\node at (-1.3,0){$c$};\node at (1,1.7){$a$};\node at (1,.35){$b$};
\node at (4,0.3){$d$};\node at (1,-1.7){$y$};\node at (1,-.35){$x$};
\node at (5,.3){$u$};\node at (6.3,0){$v$};
\filldraw (-1,0) circle(2pt);\filldraw (1,.6) circle(2pt);\filldraw (1,-.6) circle(2pt);
\filldraw (1,-1.4) circle(2pt);\filldraw (1,1.4) circle(2pt);
\filldraw (4,0) circle(2pt);\filldraw (5,0) circle(2pt);\filldraw (6,0) circle(2pt);
\end{tikzpicture}
\end{center}
\end{figure}

By the above lemma we have the following:

\begin{corollary}
The maximal Fine sets partition the vertex set of a graph with diameter greater than $2$.
\end{corollary}

\begin{definition}Let $G$ be a graph with $diam(G)\geqslant3$.  $\hat{G}$ be the graph with $$V(\hat{G})=\{\hat{x}:\hat{x}\text{ is the maximal Fine set of } G\text{ containing } x\}$$ and two distinct vertices $\hat{x},\hat{y} \in V(\hat{G})$ are adjacent if and only if $x$ and $y$ are adjacent in $G$.

\begin{lemma}\label{777}
Let $G$ be a graph with $diam(G)\geqslant3$ and 
 $\hat{x},\hat{y}\in V(\hat{G})$ are distinct vertices, where $x,y\in V(G)$. Then $d(x,y)=d(\hat{x},\hat{y})$.
\end{lemma}
\begin{proof}
Let $\hat{x}=\hat{x_0}\hat{x_1}\hat{x_2}\cdots \hat{x_t}=\hat{y}$ be a $\hat{x}\hat{y}$-path. By the properties of Fine sets $x=x_0x_1x_2\cdots x_t=y$ is a $xy$-path.  Then $d(x,y)\leqslant d(\hat{x},\hat{y})$. On the other hand, if $x=x_0x_1x_2\cdots x_t=y$ be a $xy$-path, then $\hat{x}=\hat{x_0}\hat{x_1}\hat{x_2}\cdots \hat{x_t}=\hat{y}$ is a $\hat{x}\hat{y}$-walk and then contains a $\hat{x}\hat{y}$-path such that its length is less than $t$. Consequently $d(\hat{x},\hat{y})\leqslant d(x,y)$ and then $d(\hat{x},\hat{y})=d(x,y)$.
\end{proof}
\end{definition}

One can see that if $x,y\in V(G)$ and $\hat{x}=\hat{y}$, then $d(x,y)\leqslant2$ and the following is hold:

\begin{corollary}
Let $G$ be a graph with $diam(G)\geqslant3$. Then, $diam(G)=diam(\hat{G})$.
\end{corollary}

\begin{remark}\label{rem213}
Let $G$ be a graph with $diam(G)\geqslant3$. Then, $diam(\hat{G})\geqslant3$. 
Thus we can consider $\hat{\hat G}$. But one can see that if $S\subseteq \hat G$ be a Fine set, then $S_1=\{a:\hat a\in S\}$ is a Fine set of $G$, which mean that $S$ has a unique element an then $\hat{\hat G}=\hat G$.
\end{remark}

\begin{theorem}\label{888}
Let $G$ be a graph with $diam(G)\geqslant3$. Then, $D_2(G)$ is connected if and only if $D_2(\hat{G})$ is connected. 
\end{theorem}

\begin{proof}
First, let $D_2(\hat{G})$ is connected and $x,y\in V(G)$. We consider the following two cases:\\
\textbf{Case 1:} 
$\hat{x}\ne\hat{y}$. There is a $\hat{x}\hat{y}$-path, $\hat{x}=\hat{x_1}\hat{x_2}\cdots\hat{x_t}=\hat{y}$ in $D_2(\hat{G})$. Since $\hat{x_i}\ne\hat{x_{i+1}}$, by Lemma \ref{777}, $d_G(x_i,x_{i+1})=d_{\hat{G}}(\hat{x_i},\hat{x_{i+1}})=2$, and then $x=x_1x_2\cdots x_t=y$ is a $xy$-path in $D_2(G)$.
\\\textbf{Case 2:} 
$\hat{x}=\hat{y}$. Since $diam(G)\geqslant3$, there is $z\in V(G)$ such that $\hat{x}\ne\hat{z}$. Then, $\hat{y}\ne\hat{z}$ and by Case 1, there are some paths between $x$ and $y$ to $z$ in $D_2(G)$.\\
 Consequently, $D_2(G)$ is connected.

Now, assume that $D_2(G)$ is connected and $\hat{x}\ne\hat{y}\in V(\hat{G})$ and  $x=x_1x_2\cdots x_t=y$ is a $xy$-path in $D_2(G)$. Since $\hat{x_i}=\hat{x_{i+1}}$ or $d_{\hat{G}}(\hat{x_i},\hat{x_{i+1}})=2$, then $\hat{x}=\hat{x_1}\hat{x_2}\cdots\hat{x_t}=\hat{y}$ is a  $\hat{x}\hat{y}$-walk  and contains a  $\hat{x}\hat{y}$-path in $D_2(\hat{G})$. Thus $D_2(\hat{G})$ is connected.
\end{proof}

By similarly argument, one can see that:
 \begin{corollary}
Let $G$ be a graph with $diam(G)\geqslant3$. Then, $diam(D_2(G))=diam(D_2(\hat{G}))$.
\end{corollary}

Now, we are ready to state and prove the main result of the paper.

\begin{theorem}
If $G$ be a connected graph and $diam(G)\geqslant3$, then $D_2(G)$ is disconnected if and only if $\hat{G}$ is a bipartite graph.
\end{theorem}

\begin{proof}
If $\hat{G}$ is a bipartite graph, then it is trivial that $D_2(\hat{G})$ is disconnected and by Theorem \ref{888}, $D_2(G)$ is disconnect.\\
Now, by the contrary, let $D_2(G)$ is disconnected but $\hat{G}$ contains an odd cycle. Let $\hat{H}\leqslant\hat{G}$ where $\hat{H}$ is a non-trival induced connected subgraph with $V(\hat{H})=\{\hat{x}_1, \hat{x}_2, \cdots, \hat{x}_t\}$   and $D_2(\hat{H})$ is connected. Obviously, $H=\langle\{x_1, x_2, \cdots , x_t\}\rangle$ is a connected induced subgraph of $G$ and $D_2(H)$ is connected too. By the proof of Lemma \ref{999}, there is an induced subgraph $K$ which containing $H$ and $[V(K),N_G(V(K))\setminus V(K)]$ is a subgraph of $G$. Then, $V(K)$ is Fine and then there is $x\in V(G)$ such that $V(H)\subseteq \hat{x}$, a contradiction.

Thus, $\hat{G}$ has no connected induced subgraph $\hat{H}$ with connected 2-distance graph.
Consequently, by Lemma \ref{23}, $\hat{G}$ has no induced subgraph belongs to  $\mathcal{H}$, specially, $\hat{G}$ has no odd cycle of length greater than $3$.\\
Let $\hat{G}$ has a triangle.\\
Claim 1: For any triangle $\hat{a}\hat{b}\hat{c}\hat{a}$ in $\hat{G}$, $\langle N(\hat a)\rangle$ is connected.
\\Proof: By a contrary, let $\langle N(\hat a)\rangle$ is disconnected and $H$ is the connected component of  $\langle N(\hat a)\rangle$ containing the edge $\hat b\hat c$. Let $\hat y \in N(\hat a)\setminus V(H)$. If $N(V(H))\subseteq N(\hat a)$, then $N(V(H))\setminus V(H)=\{\hat a\}$. Therefore, $V(H)$ is Fine and containing $\hat b$ and $\hat c$. By Remark \ref{rem213}, we have $\hat{b}=\hat{c}$, a contradiction. Then, we assume that  $\hat u\in N(V(H))\setminus N(\hat a)$. If $V(H)\not\subseteq N(\hat u)$ then since $H$ is connected, there are $\hat v,\hat w\in V(H)$ such that $\hat v\hat w\in H$ and $\hat v\in N(\hat u)$ and $\hat w\notin N(\hat u)$. Consequently, one can see that $\langle \{\hat y, \hat a,\hat v , \hat w,\hat u\}\rangle\in\mathcal{H}$, a contradiction. Thus, $V(H)\subseteq N(\hat u)$ for all $\hat u\in N(V(H))\setminus N(\hat a)$ and again $V(H)$ is Fine, which contradict to Remark \ref{rem213}. \\
Claim 2: Any vertex of $\hat{G}$ is on a triangle. \\
Proof: By contrary, there is a vertex $\hat{a}$ lies on no triangle. Since $\hat G$ is connected and has at least a triangle, then there is a vertex $\hat b$ and triangle $\hat x\hat y\hat z\hat x$ such that $\hat b\in N(\hat x)$ and $\hat{G}$ has no triangle containing $\hat b$, that is, $N(\hat b)\cap N(\hat x)=\emptyset$. Now, $\langle N(\hat x)\rangle$ and consequently $\langle N(x)\rangle$ is not connected, a contradiction by Claim 1.\\
Claim 3: $\hat G$ has a subgraph $\hat H$ with vertex set $V(\hat H)=\{\hat a, \hat b, \hat x, \hat c ,\hat d\}$ and $E(\hat H)=\{\hat a\hat b, \hat a\hat x, \hat b\hat x, \hat c\hat x, \hat b\hat c, \hat c\hat d\}$ where $d(\hat a,\hat d)=3$.\\
\begin{figure}[h]\label{fig1}
\begin{center}
\begin{tikzpicture}
\draw (1,-.7)--(-1,0)--(1,.7)--(3,0)--(5,0);
\draw (1,.7)--(1,-.7)--(3,0);
\filldraw (-1,0) circle(2pt);\filldraw (1,-.7) circle(2pt);
\filldraw (1,.7) circle(2pt);\filldraw (3,0) circle(2pt);\filldraw (5,00) circle(2pt);
\node at (-1.3,0) {$\hat a$};\node at (1,-1) {$\hat x$};\node at (1,1) {$\hat b$};\node at (3,-.3) {$\hat c$};\node at (5,-0.3) {$\hat d$};
%\node at  (-.3,1) $\hat a$;
\end{tikzpicture}
\end{center}
\end{figure}
Proof: Since $diam(\hat{G})\geqslant3$, there is a path $\hat a\hat b\hat c\hat d$ with $d(\hat a,\hat d)=3$. By Claim 2, $\hat b$ lies on a triangle and by claim 1, $\langle N(\hat b)\rangle$ is connected. Then there is $\hat z\in N(\hat b)\cap N(\hat c)$. If $\hat z\notin N(\hat a)\cup N(\hat d)$, then $\langle \{\hat a, \hat b, \hat z, \hat c, \hat d\}\rangle\in\mathcal{H}$, a contradiction. Therefore, $z\in N(\hat a)\cup N(\hat d)$ and by symmetry, we can assume $\hat z$ is adjacent to $\hat a$ and the proof is completed.

Now we have all the necessary information to continue the argument

 Let $H$ be the subgraph constructed in Claim 3. Let $K$ be a connected component of $\langle N(\hat a)\cap N(\hat c)\rangle$ containing the edge $bx$. Let $S_1=N(\hat a)\cap N(\hat c)\setminus V(K)$ and $S_2=N(\hat a)\setminus N(\hat c)$. We claim that $V(K)$ is Fine. If $\hat y\in N(V(K))\setminus N(\hat a)$ and $V(K)\not\subseteq N(\hat y)$, Since $K$ is connected, then there is $\hat u,\hat v\in V(K)$ such that $\hat u\hat v,\hat u\hat y\in E(\hat{G})$ but $\hat v\notin N(\hat y)$.

The subgraph $\langle \{\hat a,\hat c, \hat d, \hat u, \hat v, \hat y\}\rangle$ can be one of the four graphs shown in  the below.
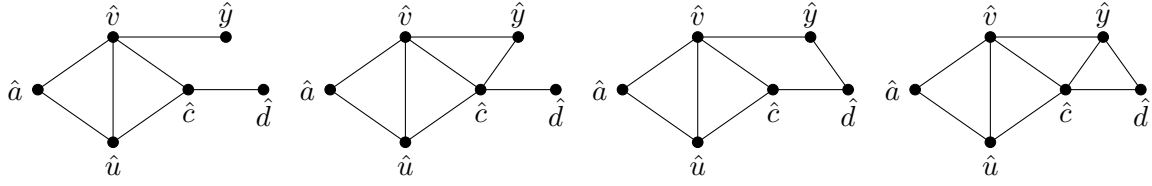
\begin{center}
\begin{figure}[h]\label{fig2}
\begin{tikzpicture}
\draw (1,-.7)--(0,0)--(1,.7)--(2,0)--(3,0);
\draw (2.5,.7)--(1,.7)--(1,-.7)--(2,0);
\filldraw (0,0) circle(2pt);\filldraw (1,-.7) circle(2pt);\filldraw (2.5,.7) circle(2pt);
\filldraw (1,.7) circle(2pt);\filldraw (2,0) circle(2pt);\filldraw (3,00) circle(2pt);
\node at (-.3,0) {$\hat a$};\node at (1,-1) {$\hat u$};\node at (1,1) {$\hat v$};\node at (2,-.3) {$\hat c$};\node at (3,-0.3) {$\hat d$};\node at (2.5,1) {$\hat y$};
%\node at  (-.3,1) $\hat a$;
\end{tikzpicture}
\begin{tikzpicture}
\draw (1,-.7)--(0,0)--(1,.7)--(2,0)--(3,0);
\draw (2,0)--(2.5,.7)--(1,.7)--(1,-.7)--(2,0);
\filldraw (0,0) circle(2pt);\filldraw (1,-.7) circle(2pt);\filldraw (2.5,.7) circle(2pt);
\filldraw (1,.7) circle(2pt);\filldraw (2,0) circle(2pt);\filldraw (3,00) circle(2pt);
\node at (-.3,0) {$\hat a$};\node at (1,-1) {$\hat u$};\node at (1,1) {$\hat v$};\node at (2,-.3) {$\hat c$};\node at (3,-0.3) {$\hat d$};\node at (2.5,1) {$\hat y$};
%\node at  (-.3,1) $\hat a$;
\end{tikzpicture}
\begin{tikzpicture}
\draw (1,-.7)--(0,0)--(1,.7)--(2,0)--(3,0);
\draw (3,0)--(2.5,.7)--(1,.7)--(1,-.7)--(2,0);
\filldraw (0,0) circle(2pt);\filldraw (1,-.7) circle(2pt);\filldraw (2.5,.7) circle(2pt);
\filldraw (1,.7) circle(2pt);\filldraw (2,0) circle(2pt);\filldraw (3,00) circle(2pt);
\node at (-.3,0) {$\hat a$};\node at (1,-1) {$\hat u$};\node at (1,1) {$\hat v$};\node at (2,-.3) {$\hat c$};\node at (3,-0.3) {$\hat d$};\node at (2.5,1) {$\hat y$};
%\node at  (-.3,1) $\hat a$;
\end{tikzpicture}
\begin{tikzpicture}
\draw (1,-.7)--(0,0)--(1,.7)--(2,0)--(3,0)--(2.5,.7);
\draw (2,0)--(2.5,.7)--(1,.7)--(1,-.7)--(2,0);
\filldraw (0,0) circle(2pt);\filldraw (1,-.7) circle(2pt);\filldraw (2.5,.7) circle(2pt);
\filldraw (1,.7) circle(2pt);\filldraw (2,0) circle(2pt);\filldraw (3,00) circle(2pt);
\node at (-.3,0) {$\hat a$};\node at (1,-1) {$\hat u$};\node at (1,1) {$\hat v$};\node at (2,-.3) {$\hat c$};\node at (3,-0.3) {$\hat d$};\node at (2.5,1) {$\hat y$};
%\node at  (-.3,1) $\hat a$;
\end{tikzpicture}
\end{figure}
\end{center}
One can see that in each case, $\langle \{\hat a,\hat c, \hat d, \hat u, \hat v, \hat y\}\rangle$ contains an induced subgraph of $\hat G$ which  belongs to $\mathcal{H}$, a contradiction. 

If $\hat y\in S_2\cap N(V(K))$ and $V(K)\not\subseteq N(\hat y)$  , then there are $\hat u\hat v\in E(K)$ such that $\hat u\hat y, \hat u\hat v\in E(\hat{G})$ but $\hat v\notin N(\hat y)$. Then, $\langle \{\hat y, \hat u, \hat v, \hat c, \hat d\}\rangle\in\mathcal{H}$, a contradiction. Thus, $V(K)$ is Fine, again is a contradiction. Therefore $\hat{G}$ has no triangle which complete the proof.
\end{proof}

\end{document}